\documentclass[a4paper,11pt, leqno]{amsart}
\usepackage{amssymb}
\usepackage{amsthm}
\usepackage{mathrsfs, amsfonts, amsmath}
\usepackage{xcolor}
\usepackage{graphicx}
\usepackage{tikz}

\usepackage[alphabetic,initials,nobysame]{amsrefs}

\newcommand{\diam}{\operatorname{diam}}
\newcommand{\dist}{\operatorname{dist}}

\newcommand{\modu}{\operatorname{mod}}

\newcommand{\hatc}{\hat{\mathbb{C}}}

\newcommand{\Le}{\operatorname{Le}}
\newcommand{\Ri}{\operatorname{Ri}}
\newcommand{\Up}{\operatorname{Up}}
\newcommand{\Do}{\operatorname{Do}}

\newcommand{\capa}{\operatorname{cap}}

\numberwithin{equation}{section} \theoremstyle{plain}
\newtheorem{theorem}{\textbf{THEOREM}}[section]
\newtheorem{lemma}[theorem]{\textsc{Lemma}}
\newtheorem{proposition}[theorem]{\textsc{Proposition}}
\theoremstyle{definition}
{\theoremstyle{remark} \newtheorem{remark}[theorem]{Remark}}
\def\charfn_#1{{\raise1.2pt\hbox{$\chi_{\kern-1pt\lower3pt\hbox{{$\scriptstyle#1$}}}$}}}
\def\leq{\leqslant }
\def\geq{\geqslant }

\def\XXint#1#2#3{{\setbox0=\hbox{$#1{#2#3}{\int}$}
\vcenter{\hbox{$#2#3$}}\kern-.5\wd0}}

\begin{document}

\title{Rigid circle domains with non-removable boundaries}
\author{Kai Rajala} 
	\address{Department of Mathematics and Statistics, University of Jyv\"askyl\"a, P.O. Box 35 (MaD), FI-40014, University of Jyv\"askyl\"a, Finland.}
	\email{kai.i.rajala@jyu.fi}
\date{}

\thanks{  
\newline Research supported by the Research Council of Finland, project number 360505. }

\maketitle

\begin{center}
   \vspace{.3cm}
    \textit{Dedicated to Jang-Mei Wu.}
   \vspace{.3cm}
\end{center}

\begin{abstract} 
We give a negative answer to the \emph{rigidity conjecture} of He and Schramm by constructing a rigid circle domain $\Omega$ on the Riemann sphere $\hatc$ with conformally 
non-removable boundary. Here rigidity means that every conformal map from $\Omega$ onto another circle domain is a M\"obius transformation, and non-removability means that 
there is a homeomorphism of $\hatc$ which is conformal on $\hatc \setminus \partial \Omega$ but not everywhere. 

Our construction is based on a theorem of Wu, which states that the product of any Cantor set $E$ with a sufficiently thick Cantor set $F$ is non-removable. We show that one can choose   
$E$ and $F$ so that the complement of the union of $E \times F$ and suitably placed disks is rigid. 

The proof of rigidity involves a metric characterization of conformal maps, which was recently proved by Ntalampekos. The other direction of the rigidity conjecture, i.e., whether removability of the boundary implies rigidity, remains 
open. 
\end{abstract}

\renewcommand{\baselinestretch}{1.2}


\section{Introduction}
A subdomain $\Omega$ of the Riemann sphere $\hatc$ is a \emph{circle domain} if every connected component of $\partial \Omega$ is a circle or a point. The long-standing 
\emph{Koebe conjecture} \cite{Koe08} asserts that every subdomain of $\hatc$ admits a conformal map $f$ onto a circle domain. Koebe proved that every \emph{finitely connected} 
domain satisfies the conjecture and that $f$ is unique up to postcomposition by a M\"obius transformation. 

Uniqueness is equivalent to \emph{ridigity}: a circle domain $\Omega$ is \emph{(conformally) rigid}, if every conformal map $f:\Omega \to \Omega'$ onto another circle domain 
is the restriction of a M\"obius transformation. Complements of Cantor sets $K$ with positive area are basic examples of non-rigid circle domains; solving the \emph{Beltrami 
equation} (see e.g. \cite{Ahl66}, \cite{AIM09}) with coefficient $\mu = \frac{1}{2}\chi_K$ yields a quasiconformal homeomorphism which is conformal only in $\hatc \setminus K$. 

In a breakthrough work \cite{HeSch93}, He and Schramm applied the rigidity of \emph{countably connected circle domains} to verify Koebe's conjecture for all   
countably connected domains. In \cite{HeSch94}, they moreover proved the rigidity of circle domains whose boundary has \emph{$\sigma$-finite length}. Further sufficient conditions for rigidity were 
established by Ntalampekos and Younsi in \cite{NtaYou20}, \cite{You16}, and \cite{Nta23a} (see also \cite{BKM09}, \cite{Bon11}, \cite{Mer12}). 

Towards a characterization of rigid circle domains $\Omega$, He and Schramm \cite{HeSch94} pointed out connections to \emph{conformal removability}, and conjectured that rigidity of $\Omega$ 
is equivalent to the conformal removability of $\partial \Omega$. Here a compact $K \subset \hatc$ is \emph{conformally removable} (or \emph{CH-removable}), if every homeomorphism 
$h:\hatc \to \hatc$, which is conformal on $\hatc \setminus K$, is a M\"obius transformation. Our main result gives a negative answer.

\begin{theorem}
\label{thm:main}
There is a rigid circle domain $\Omega \subset \hatc$ such that $\partial \Omega$ is conformally non-removable. 
\end{theorem}

The proof below shows that the answer to another version of the rigidity conjecture given in \cite{HeSch94}, which asks if rigidity is equivalent to the conformal removability of 
every Cantor set contained in $\partial \Omega$, is also negative. 

It follows from the definitions that if the boundary of a rigid circle domain \emph{is} a Cantor set $K$, then $K$ is conformally removable. 
The other direction of the rigidity conjecture, which asks if circle domains with removable boundaries are rigid, remains open even for domains with Cantor set boundaries. 

Conformal removability is an active and challenging research topic, see e.g. \cite{JonSmi00}, \cite{You15}, \cite{Nta19}, \cite{Nta24b}, \cite{Nta24c}, and the references therein. A major difficulty is that constructing non-trivial conformal maps $f$ outside exceptional sets becomes considerably harder if one also requires the existence of a \emph{homeomorphic} extension of $f$ to $\hatc$. 

A basic example of a non-removable set is $K=E \times [0,1]$ for any Cantor set $E$: one can apply an essentially $1$-dimensional construction, starting with a continuous measure on $E$, 
to produce a non-trivial homeomorphism which is conformal off $K$. Much more involved constructions of non-removable sets were given by Kaufman \cite{Kau84}, Bishop \cite{Bis94}, and Wu \cite{Wu98}, whose result is an important ingredient of the proof of Theorem \ref{thm:main}. Here $\capa$ is the logarithmic capacity, see e.g. \cite[Ch. 9]{Pom92}. 

\begin{theorem}[\cite{Wu98}] \label{thm:Wu}
Let $E$ and $F$ be two Cantor sets in $\mathbb{R}$. If 
\begin{equation} \label{eq:logcapa}
\capa([a,b] \setminus F) < \capa([a,b]) 
\end{equation} 
for some interval $[a,b]$, then $E \times F$ is conformally non-removable.  
\end{theorem} 

Although the proof of Theorem \ref{thm:Wu} is subtle, the rough idea is similar to the case $E \times [0,1]$ above. Namely, by Ahlfors and Beurling \cite{AhlBeu50}, Condition \eqref{eq:logcapa} yields a non-trivial conformal embedding $f:\hatc \setminus (\{0\} \times F) \to \hatc$. Such an $f$ cannot admit a continuous 
extension to $\hatc$. However, given a continuous probability measure $\mu$ on $E$, one can produce a global homeomorphism that is conformal off $E \times F$, by considering averages of the map $f$ with 
respect to $\mu$ in the real variable. Thus, $E \times F$ is non-removable. 

To prove Theorem \ref{thm:main} we choose a thick Cantor set $F$ satisfying \eqref{eq:logcapa}, a thin Cantor set $E$, and closed disks $D_j$. We let $\Omega=\hatc \setminus ((E \times F) \cup (\cup_j D_j))$,  
and show that disks $D_j$ can be placed so that every conformal map $f:\Omega \to \Omega'$ between circle domains must have bounded \emph{eccentric distortion}. Therefore we can, after extending $f$ to a homeomorphism of $\hatc$ using a familiar reflection (Schottky group) method, apply a recent result of Ntalampekos \cite{Nta24a} to prove that $f$ is a M\"obius transformation. Thus $\Omega$ is rigid. But since $\partial \Omega \supset E \times F$, Theorem \ref{thm:Wu} shows that $\partial \Omega$ is conformally non-removable.


\section{Proof of Theorem \ref{thm:main}: Construction of $\Omega$} \label{sec:const}
To start the proof of Theorem \ref{thm:main}, we construct the complement of a circle domain $\Omega$ by using the basic building block in Section \ref{sec:block} followed by an iteration procedure in Section \ref{sec:itera}. 

\subsection{Basic building block} \label{sec:block}

We fix an even integer $N \geq 2$ and let $I_*^n$, $n \in \{1,\ldots,N\}$, be closed segments with equal length $\ell(I_*^n)=:2s$ obtained by removing $N-1$ open segments of length $a$ from $[0,1]$. The segments 
are ordered so that $I_*^1$ contains $0$ and $I_*^N$ contains $1$. 

We remove another open segment of length $a$ from the middle of each $I_*^n$ to obtain segments $J_*^n(\Do)$ and $J_*^n(\Up)$ of equal length $s-\frac{a}{2}$, 
containing the left, respectively right, endpoint of $I_*^n$. 

Next, we fix $\epsilon>0$ and define the following subsets of $\mathbb{C}$ for $n \in \{1,\ldots,N\}$:  
\begin{eqnarray}
\nonumber I^n(\Le) &=&-\epsilon+i I_*^n, \, I^n(\Ri)=\epsilon+iI_*^n, \\ 
\nonumber J^n(\Le,\Do)&=&-\epsilon+ iJ_*^n(\Do), \, J^n(\Le,\Up)=-\epsilon+iJ_*^n(\Up), \\
\nonumber J^n(\Ri,\Do)&=&\epsilon+iJ_*^n(\Do), \, J^n(\Ri,\Up)=\epsilon+iJ_*^n(\Up). 
\end{eqnarray} 
We denote by $R^n$ the closed rectangle whose vertical sides are the segments $I^n(\Le)$ and $I^n(\Ri)$, and by $iy^n$ the center of $R^n$. 
Finally, let $D^n(\Le)$ and $D^n(\Ri)$ be the closed disks with radius $s$ and centers $-2\epsilon-s+iy^n$ and $2\epsilon+s+iy^n$, respectively.


\subsection{Iteration} \label{sec:itera} 

Given the sequence of integers $N_j$ defined below, we denote by $\mathcal{N}_k$ the collection of words of length $k$ 
with letters $n_j \in \{1,\ldots,N_j\}$, i.e., 
$$
\mathcal{N}_k=\Big\{\tilde{n}=n_1n_2\cdots n_k:\, n_j \in \{1,\ldots,N_j\} \text{ for all } j \in \{1,\ldots,k\} \Big\}. 
$$
We also denote $\mathcal{W}_k=\mathcal{M}_k \times \mathcal{N}_k$, where 
$$ 
\mathcal{M}_k=\Big\{\tilde{m}=m_1m_2\cdots m_k: \,  m_j \in \{\Le,\Ri\} \text{ for all } j \in \{1,\ldots,k\} \Big\}. 
$$ 
We apply the construction in Section \ref{sec:block} with parameters $N=N_1,a=a_1,s=s_1,\epsilon=\epsilon_1$ satisfying 
\begin{equation} \label{eq:inichoice}
N_1=2, \quad \epsilon_1=\frac{a_1}{100}=\frac{s_1}{10^5}. 
\end{equation} 
We obtain rectangles $R^1,R^2$, as well as the other sets defined above. We complete the first step of the construction by ``duplicating'', i.e., if $n_1 \in \{1,2\}$ we 
define 
\begin{eqnarray*} 
R^w= \left\{ \begin{array}{ll}
-2+R^{n_1}, & w=(\Le,n_1) \in \mathcal{W}_1, \\ 
2+R^{n_1}, & w=(\Ri,n_1) \in \mathcal{W}_1, 
\end{array}
\right. 
\end{eqnarray*}
and use similar notation for the other sets constructed. Altogether, after the first step we have two copies of the sets constructed in Section \ref{sec:block}; one on the left half-plane and another one on the right half-plane, e.g., $J^{(\Le,2)}(\Le,\Up)=J^2(\Le,\Up)-2$ and $J^{(\Ri,2)}(\Le,\Up)=J^2(\Le,\Up)+2$. See Figure \ref{fig:kuva1} for an illustration. 


\begin{figure}
    \centering
    
\begin{tikzpicture}
\draw (-4.05,-2.3) -- (-4.05,-1.3);
\draw (-4.05,-1.1) -- (-4.05,-.1);
\draw (-4.05,.1) -- (-4.05,1.1);
\draw (-4.05,1.3) -- (-4.05,2.3);

\draw (-3.95,-2.3) -- (-3.95,-1.3);
\draw (-3.95,-1.1) -- (-3.95,-.1);
\draw (-3.95,.1) -- (-3.95,1.1);
\draw (-3.95,1.3) -- (-3.95,2.3);

\draw (4.05,-2.3) -- (4.05,-1.3);
\draw (4.05,-1.1) -- (4.05,-.1);
\draw (4.05,.1) -- (4.05,1.1);
\draw (4.05,1.3) -- (4.05,2.3);

\draw (3.95,-2.3) -- (3.95,-1.3);
\draw (3.95,-1.1) -- (3.95,-.1);
\draw (3.95,.1) -- (3.95,1.1);
\draw (3.95,1.3) -- (3.95,2.3);

\draw (-5.2,-1.2)  circle (1.1); 
\draw (-5.2,1.2)  circle (1.1); 
\draw (-2.8,-1.2)  circle (1.1); 
\draw (-2.8,1.2)  circle (1.1); 

\draw (5.2,-1.2)  circle (1.1); 
\draw (5.2,1.2)  circle (1.1); 
\draw (2.8,-1.2)  circle (1.1); 
\draw (2.8,1.2)  circle (1.1); 
\end{tikzpicture}

    \caption{First step of the construction. }
    \label{fig:kuva1}
\end{figure}
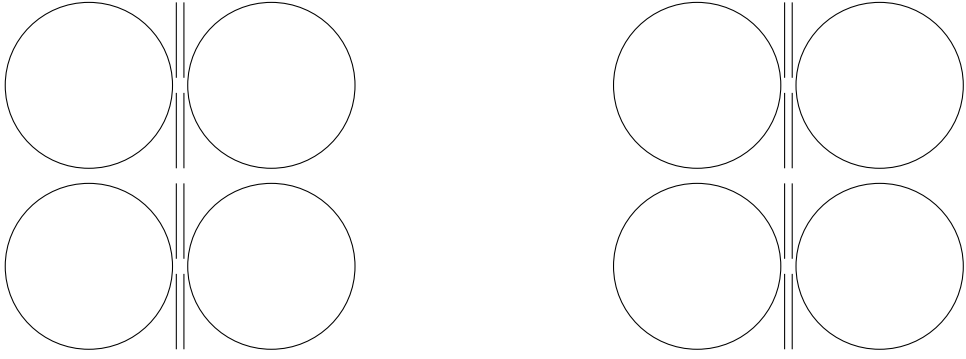


We then assume that rectangles 
\begin{equation} \label{eq:rectcoord}
R^{w}= [x^{\tilde{m}}-t_{k-1},x^{\tilde{m}}+t_{k-1}] \times [y^{\tilde{n}}-r_{k-1},y^{\tilde{n}}+r_{k-1}] 
\end{equation} 
and disks $D^w(\Le),D^w(\Ri)$ with radius $r_{k-1}$ and centers 
\begin{equation} \label{eq:centers} 
x^{\tilde{m}} \pm(2t_{k-1}+r_{k-1})+iy^{\tilde{n}}
\end{equation} 
have been constructed for all $w=(\tilde{m},\tilde{n}) \in \mathcal{W}_{k-1}$, $k \geq 2$. The coordinate $x^{\tilde{m}}$ does not depend on $\tilde{n}$, and the coordinate $y^{\tilde{n}}$ does not depend on $\tilde{m}$. 
By the first step above  we have $r_1=s_1$, $x^{\Le}=-2$, $x^{\Ri}=2$, and $t_1=\epsilon_1$. 

Intervals $J^w(\cdot,\Do)$ and $J^w(\cdot,\Up)$ are obtained 
by removing an open interval of length $\delta_{k-1}$ from $I^w(\cdot)$;  
\begin{eqnarray} \nonumber
I^w(\cdot) &=& x^{\tilde{m}} \pm t_{k-1}+ i[y^{\tilde{n}}-r_{k-1},y^{\tilde{n}}+r_{k-1}], \\ 
\label{eq:wilding} J^w(\cdot,\Do) &=& x^{\tilde{m}} \pm t_{k-1}+i\left[y^{\tilde{n}}-r_{k-1},y^{\tilde{n}}-\frac{\delta_{k-1}}{2}\right], \\ 
\nonumber J^w(\cdot,\Up) &=& x^{\tilde{m}} \pm t_{k-1}+i\left[y^{\tilde{n}}+\frac{\delta_{k-1}}{2},y^{\tilde{n}}+r_{k-1}\right]. 
\end{eqnarray}

We fix such a $w$. Our goal is to construct the segments, rectangles and disks corresponding to all the ``children'' $w'=(\tilde{m}m_{k},\tilde{n}n_k) \in \mathcal{W}_k$ of $w$. We denote by 
$\phi_{\Do}^w$ and $\phi^w_{\Up}$ the homotheties (i.e., maps $z \mapsto \alpha z +\beta$, where $\alpha>0$ and $\beta \in \mathbb{C}$) for which 
\begin{eqnarray*}
& &  \phi^w_{\Do}\big(x^{\tilde{m}}+i(y^{\tilde{n}}-r_{k-1})\big)=0 \quad \text{and} \quad \phi_{\Do}^w\Big(x^{\tilde{m}}+i(y^{\tilde{n}}-\frac{\delta_{k-1}}{2})\Big) = i, \\ 
& & \phi_{\Up}^w\Big(x^{\tilde{m}}+i(y^{\tilde{n}}+\frac{\delta_{k-1}}{2})\Big) = 0 \quad \text{and} \quad \phi^w_{\Up}\big(x^{\tilde{m}}+i(y^{\tilde{n}}+r_{k-1})\big)=i. 
\end{eqnarray*} 
Then there is $\hat{\epsilon}_{k-1} >0$ so that 
\begin{eqnarray*}
\begin{split}
\phi_{\Do}^w(J^w(\Le,\Do)) &= \phi_{\Up}^w(J^w(\Le,\Up)) 
= -\hat{\epsilon}_{k-1} + i[0,1] \quad \text{and } \\
\phi_{\Do}^w(J^w(\Ri,\Do)) &= \phi_{\Up}^w(J^w(\Ri,\Up)) 
= \hat{\epsilon}_{k-1} + i[0,1]. 
\end{split}
\end{eqnarray*}

Let $N=N_k$ and $a=a_k,s=s_k,\epsilon=\epsilon_k>0$ be the numbers for which 
\begin{equation} 
\label{eq:seas}
N_k =2\left\lceil \frac{100}{\hat{\epsilon}_{k-1}} \right\rceil, \quad \frac{a_k}{s_k}=\frac{1}{\exp((2N_k)^{2k})}, \quad   \epsilon_k =\min\left\{\frac{\hat{\epsilon}_{k-1}}{100},\frac{a_k}{100}\right\}. 
\end{equation} 
The idea behind the construction of $\Omega$ is that we can first choose $N_k$ (number of new rectangles $R^{w'}$) to be as large as we wish, then $\frac{a_k}{s_k}$ (relative size of removed segments) as small as we wish, and finally $\epsilon_k$ (relative width) as small as we wish. The requirement on $N_k$ and the first upper bound on 
$\epsilon_k$ in \eqref{eq:seas} guarantee that the sets defined below are disjoint subsets of $R^w$. 

The requirement on $\frac{a_k}{s_k}$ guarantees 
that the holes between the vertical segments are thin. Such thinness will lead to a thick Cantor set $F$ (on the imaginary axis) satisfying Condition \eqref{eq:logcapa}, and to the non-removability of $\partial \Omega$. The second upper bound on $\epsilon_k$ will lead to a Cantor set $E$ (on the real axis), which ``is thinner than $F$ is thick''. The precise meaning of 
such thinness will be given in terms of conformal modulus estimates in Section \ref{sec:modusec}, which will be applied to prove the ridigity of $\Omega$. 

Applying the construction in Section \ref{sec:block} with the parameters $\frac{N_k}{2},a_k,s_k,\epsilon_k$ defined in \eqref{eq:seas} yields rectangles $R_*:=R_*^{\tilde{n}n_k}$ and disks 
$D_*^{\tilde{n}n_k}(\Le), D_*^{\tilde{n}n_k}(\Ri)$ for all $n_k \in \{1,\ldots,\frac{N_k}{2}\}$. Finally, given $w'=(\tilde{m}m_k,\tilde{n}n_k)$, we define  
\begin{eqnarray*}
R^{w'}=(\phi_\kappa^w)^{-1}(R_*+\lambda), 
\end{eqnarray*} 
where 
\begin{eqnarray*} 
\lambda = \left\{\begin{array}{ll}  \frac{-\hat{\epsilon}_{k-1}}{2} & m_k=\Le,  \\  \frac{\hat{\epsilon}_{k-1}}{2}, & m_k=\Ri,    
\end{array} \right.
\end{eqnarray*} 
and 
\begin{eqnarray*} 
\kappa = \left\{\begin{array}{ll} \Do, & n_k \in \{1,\ldots,\frac{N_k}{2} \}, \\  \Up, & n_k \in \{\frac{N_k}{2}+1,\ldots,N_k\}. 
\end{array} \right. 
\end{eqnarray*}
Disks $D^{w'}(\cdot)$ and segments $I^{w'}(\cdot),J^{w'}(\cdot,\cdot)$ are defined in a similar manner. See Figure \ref{fig:kuva2} for an illustration. 


\begin{figure}
    \centering
    
\begin{tikzpicture}
\draw (-4.4,-2.5) -- (-.4,-2.5);
\draw (-4.4,2.5) -- (-.4,2.5);
\draw (4.4,-2.5) -- (.4,-2.5);
\draw (4.4,2.5) -- (.4,2.5);

\draw (-4.4,-1.23) -- (-4,-1.23); 
\draw (-3.885,-1.23) -- (-3.485,-1.23); 
\draw (-3.37,-1.23) -- (-2.97,-1.23); 
\draw (-2.85,-1.23) -- (-2.45,-1.23); 
\draw (-2.34,-1.23) -- (-1.94,-1.23); 
\draw (-1.82 ,-1.23) -- (-1.42,-1.23); 
\draw (-1.3,-1.23) -- (-.9,-1.23); 
\draw (-.8,-1.23) -- (-.4,-1.23); 

\draw (4.4,-1.23) -- (4,-1.23); 
\draw (3.885,-1.23) -- (3.485,-1.23); 
\draw (3.37,-1.23) -- (2.97,-1.23); 
\draw (2.85,-1.23) -- (2.45,-1.23); 
\draw (2.34,-1.23) -- (1.94,-1.23); 
\draw (1.82 ,-1.23) -- (1.42,-1.23); 
\draw (1.3,-1.23) -- (.9,-1.23); 
\draw (.8,-1.23) -- (.4,-1.23); 

\draw (-4.4,-1.17) -- (-4,-1.17); 
\draw (-3.885,-1.17) -- (-3.485,-1.17); 
\draw (-3.37,-1.17) -- (-2.97,-1.17); 
\draw (-2.85,-1.17) -- (-2.45,-1.17); 
\draw (-2.34,-1.17) -- (-1.94,-1.17); 
\draw (-1.82 ,-1.17) -- (-1.42,-1.17); 
\draw (-1.3,-1.17) -- (-.9,-1.17); 
\draw (-.8,-1.17) -- (-.4,-1.17); 

\draw (4.4,-1.17) -- (4,-1.17); 
\draw (3.885,-1.17) -- (3.485,-1.17); 
\draw (3.37,-1.17) -- (2.97,-1.17); 
\draw (2.85,-1.17) -- (2.45,-1.17); 
\draw (2.34,-1.17) -- (1.94,-1.17); 
\draw (1.82 ,-1.17) -- (1.42,-1.17); 
\draw (1.3,-1.17) -- (.9,-1.17); 
\draw (.8,-1.17) -- (.4,-1.17);

\draw (-4.4,1.23) -- (-4,1.23); 
\draw (-3.885,1.23) -- (-3.485,1.23); 
\draw (-3.37,1.23) -- (-2.97,1.23); 
\draw (-2.85,1.23) -- (-2.45,1.23); 
\draw (-2.34,1.23) -- (-1.94,1.23); 
\draw (-1.82 ,1.23) -- (-1.42,1.23); 
\draw (-1.3,1.23) -- (-.9,1.23); 
\draw (-.8,1.23) -- (-.4,1.23); 

\draw (4.4,1.23) -- (4,1.23); 
\draw (3.885,1.23) -- (3.485,1.23); 
\draw (3.37,1.23) -- (2.97,1.23); 
\draw (2.85,1.23) -- (2.45,1.23); 
\draw (2.34,1.23) -- (1.94,1.23); 
\draw (1.82 ,1.23) -- (1.42,1.23); 
\draw (1.3,1.23) -- (.9,1.23); 
\draw (.8,1.23) -- (.4,1.23); 

\draw (-4.4,1.17) -- (-4,1.17); 
\draw (-3.885,1.17) -- (-3.485,1.17); 
\draw (-3.37,1.17) -- (-2.97,1.17); 
\draw (-2.85,1.17) -- (-2.45,1.17); 
\draw (-2.34,1.17) -- (-1.94,1.17); 
\draw (-1.82 ,1.17) -- (-1.42,1.17); 
\draw (-1.3,1.17) -- (-.9,1.17); 
\draw (-.8,1.17) -- (-.4,1.17); 

\draw (4.4,1.17) -- (4,1.17); 
\draw (3.885,1.17) -- (3.485,1.17); 
\draw (3.37,1.17) -- (2.97,1.17); 
\draw (2.85,1.17) -- (2.45,1.17); 
\draw (2.34,1.17) -- (1.94,1.17); 
\draw (1.82 ,1.17) -- (1.42,1.17); 
\draw (1.3,1.17) -- (.9,1.17); 
\draw (.8,1.17) -- (.4,1.17); 

\draw (-3.94,-1.72) circle (.454);
\draw (-3.94+1.04,-1.72) circle (.454);
\draw (-3.94+2.08,-1.72) circle (.454);
\draw (-3.94+3.12,-1.72) circle (.454);

\draw (-3.94+4.79,-1.72) circle (.454);
\draw (-3.94+5.83,-1.72) circle (.454);
\draw (-3.94+6.87,-1.72) circle (.454);
\draw (-3.94+7.91,-1.72) circle (.454);

\draw (-3.94,-1.72+1.04) circle (.454);
\draw (-3.94+1.04,-1.72+1.04) circle (.454);
\draw (-3.94+2.08,-1.72+1.04) circle (.454);
\draw (-3.94+3.12,-1.72+1.04) circle (.454);

\draw (-3.94+4.79,-1.72+1.04) circle (.454);
\draw (-3.94+5.83,-1.72+1.04) circle (.454);
\draw (-3.94+6.87,-1.72+1.04) circle (.454);
\draw (-3.94+7.91,-1.72+1.04) circle (.454);

\draw (-3.94,-1.72+2.4) circle (.454);
\draw (-3.94+1.04,-1.72+2.4) circle (.454);
\draw (-3.94+2.08,-1.72+2.4) circle (.454);
\draw (-3.94+3.12,-1.72+2.4) circle (.454);

\draw (-3.94+4.79,-1.72+2.4) circle (.454);
\draw (-3.94+5.83,-1.72+2.4) circle (.454);
\draw (-3.94+6.87,-1.72+2.4) circle (.454);
\draw (-3.94+7.91,-1.72+2.4) circle (.454);

\draw (-3.94,-1.72+1.04+2.4) circle (.454);
\draw (-3.94+1.04,-1.72+1.04+2.4) circle (.454);
\draw (-3.94+2.08,-1.72+1.04+2.4) circle (.454);
\draw (-3.94+3.12,-1.72+1.04+2.4) circle (.454);

\draw (-3.94+4.79,-1.72+1.04+2.4) circle (.454);
\draw (-3.94+5.83,-1.72+1.04+2.4) circle (.454);
\draw (-3.94+6.87,-1.72+1.04+2.4) circle (.454);
\draw (-3.94+7.91,-1.72+1.04+2.4) circle (.454);

\end{tikzpicture}

    \caption{The next generation inside a rectangle $R^w$. The figure is rotated by $90$ degrees, and the parameters in the actual construction are different. }
    \label{fig:kuva2}
\end{figure}
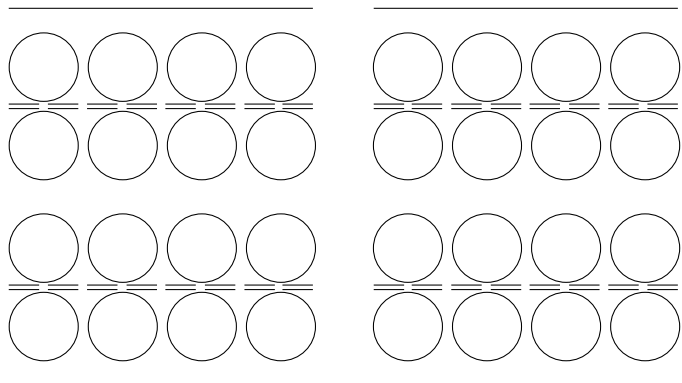


As discussed above, one can apply \eqref{eq:seas} to check that all the sets lie in $R^w$. The sets can be represented as in \eqref{eq:rectcoord}, \eqref{eq:centers}, \eqref{eq:wilding} above, changing $w$ with $w'$ and 
lengths $t_{k-1},r_{k-1},\delta_{k-1}$ with $t_k,r_k,\delta_k$. We record some consequences of \eqref{eq:inichoice} and \eqref{eq:seas} for future reference: if $k \in \mathbb{N}$, then 
\begin{equation} \label{eq:seasuusi} 
t_k \leq \frac{\delta_k}{100} \leq \frac{1}{100} \quad \text{and} \quad \delta_k \leq  \exp(-(2N_k)^{2k}) r_k\leq \exp(-(2N_k)^{2k}).  
\end{equation}
If moreover $k \geq 2$, then 
\begin{equation} \label{eq:kaksuusi} 
N_{k} \geq 20 N_{k-1}\geq 20N_1=40, \quad  r_k \leq \frac{\delta_{k-1}}{100} \leq 1 \quad \text{and} \quad t_{k}\leq \frac{t_{k-1}}{100}. 
\end{equation} 
The last inequality in \eqref{eq:kaksuusi} follows from \eqref{eq:seas} by noticing that $\frac{t_k}{t_{k-1}}=\frac{\epsilon_k}{\hat{\epsilon}_{k-1}}$. 


\subsection{Definition of $\Omega$}

We carry out the above construction for every $w \in \mathcal{W}:=\cup_k \mathcal{W}_k$, and define a circle domain 
$\Omega$ as follows:  
\begin{eqnarray*}
\nonumber \hatc \setminus \Omega &=& \Big(\bigcup_{w \in \mathcal{W}} (D^w(\Le) \cup D^w(\Ri))\Big) \cup  \bigcap_{k=1}^\infty \Big(\bigcup_{w \in \mathcal{W}_k} R^w \Big) \\ 
&=& \Big(\bigcup_{w \in \mathcal{W}} (D^w(\Le) \cup D^w(\Ri))\Big) \cup  \Big( E \times F \Big),  
\end{eqnarray*}
where $E \times F$ is the product of Cantor sets $E \subset [-3,3]$ and $F \subset [0,1]$. 
\begin{lemma} \label{lem:areazero}
The set $$E \times F=\bigcap_{k=1}^\infty \big(\bigcup_{w \in \mathcal{W}_k} R^w \big)$$ has Lebesgue measure zero.  
\end{lemma}
\begin{proof}
Given $k \in \mathbb{N}$, the Cantor set $E$ is covered by $2^k$ intervals of length $2t_{k}$. By \eqref{eq:seasuusi} we have $t_1 \leq \frac{1}{100}$, and by \eqref{eq:kaksuusi} we have $t_{k+1}\leq \frac{t_{k}}{100}$ for all $k \in \mathbb{N}$. Since $2^{k+1} 100^{-k} \to 0$ as $k \to \infty$, $E$ has zero length. The claim follows by Fubini's theorem. 
\end{proof}


\subsection{Non-removability of $\partial \Omega$ }
In this section we apply Theorem \ref{thm:Wu} to show that $\partial \Omega$ is conformally non-removable. 

\begin{theorem} \label{thm:removable}  
The boundary of $\Omega$ is conformally non-removable. 
\end{theorem} 

\begin{proof} 
We will prove Condition \eqref{eq:logcapa} for set $F$ and $[a,b]=[0,1]$, i.e., 
$$
\capa([0,1] \setminus F) < \capa([0,1]).  
$$ 
Theorem \ref{thm:Wu} then gives the desired conclusion. 
The construction of $\Omega$ can in fact be carried out so that $\capa([0,1] \setminus F)$ is smaller than any predetermined $\epsilon>0$, so crude estimates are sufficient below. 

Recall the following properties of the logarithmic capacity: 
\begin{itemize}
\item[(i)] For intervals $[c,d] \subset \mathbb{R}$ we have 
\begin{equation} 
\label{eq:interval} 
\capa([c,d]) = \frac{d-c}{4} \quad \quad \text{(see \cite[p. 207]{Pom92})}.  
\end{equation} 
\item[(ii)] If $E_\ell$ are Borel sets and if $E=\cup_{\ell=1}^\infty E_\ell$ satisfies $\operatorname{diam} E < \delta$, then 
\begin{equation} \label{eq:niisku}
\frac{1}{\log \frac{\delta}{\capa(E)}} \leq \sum_{\ell=1}^\infty \frac{1}{\log \frac{\delta}{\capa(E_\ell)}}  \quad \quad \text{(see \cite[Cor. 9.13]{Pom92})}. 
\end{equation} 
\end{itemize}

The set $[0,1] \setminus F$ is the union of the removed intervals: 
$$
[0,1] \setminus F=  \Big(\cup_{\ell=1}^3 \Delta(\ell) \Big) \cup  \Big(\bigcup_{k=2}^\infty  \Big( \bigcup_{\tilde{n} \in \mathcal{N}_{k}}  \bigcup_{\ell=1}^{2N_k-2} \Delta(\tilde{n},\ell) \Big)\Big).  
$$

Here $\Delta(\ell)$ is the projection to the imaginary axis of one of the three intervals removed from $-2-\epsilon_1+i[0,1]$ in the first step of the construction, and $\Delta(\tilde{n},\ell)$ is the projection to the imaginary axis of 
one of the $2N_k-2$ intervals removed from $I^w(\Le)$ in the $k$:th step, for any $w=(\tilde{m},\tilde{n})$. The combined cardinality of segments $\Delta(\tilde{n},\ell)$, $\tilde{n} \in \mathcal{N}_k$, is 
$$\#_k=2(N_k-1)\Pi_{j=1}^{k-1}N_j. 
$$ 
Since $N_{j} \leq N_k$ for all $j \leq k$ by \eqref{eq:kaksuusi}, it follows that  
$\#_k \leq 2N_k^k$. By \eqref{eq:seasuusi}, the length $\delta_k$ such segments is bounded from above by 
$\exp(-(2N_k)^{2k})$.

We recall that $\operatorname{diam}([0,1]\setminus F)=1$. Since $N_k \geq 20N_{k-1}$ and $N_1=2$ by \eqref{eq:kaksuusi}, we can apply the above estimates together with \eqref{eq:interval} and 
\eqref{eq:niisku} to conclude that  
\begin{equation} \label{eq:jinta}  
\frac{1}{\log \frac{2}{\capa([0,1]\setminus F)}} \leq \sum_{k=1}^\infty \frac{2N_k^k}{\log (8\exp((2N_k)^{2k})} \leq \frac{1}{4}. 
\end{equation}
Combining \eqref{eq:interval} and \eqref{eq:jinta} yields  
$$
\capa([0,1]\setminus F) \leq 2\exp(-4) < \frac{1}{4}=\capa([0,1]), 
$$
as desired. The proof is complete. 
\end{proof}


\section{Proof of Theorem \ref{thm:main}: Rigidity of $\Omega$ }
In this section we prove the rigidity of $\Omega$.  

\begin{theorem}\label{thm:rigidity} 
Domain $\Omega$ is rigid; every conformal map $f:\Omega \to \Omega'$ onto another circle domain $\Omega'$ is the restriction of a M\"obius transformation. 
\end{theorem} 
Theorem \ref{thm:main} follows from Theorems \ref{thm:removable} and \ref{thm:rigidity}. The key property towards Theorem \ref{thm:rigidity} is that we can surround every $z \in E \times F$ with unions of nearby disks and families of 
paths in $\Omega$ of large conformal modulus, see Section \ref{sec:modusec}. It follows (see Section \ref{sec:distortion}) that 
there is a sequence of disks $D^w(z)$ whose relative distances to $z$ are small both in the domain and 
after mapping with any conformal map from $\Omega$ onto another circle domain. 
Combining with the extension procedure in Sections \ref{sec:boundary} and \ref{sec:Schottky}, and with Ntalampekos' metric characterization of conformal maps 
(see Section \ref{sec:final}) shows that $f$ must be the restriction of a M\"obius transformation. 


\subsection{Conformal modulus}\label{sec:modudef}
The \emph{conformal modulus} $\modu(\Gamma)$ of a family of paths $\Gamma$ in $\mathbb{C}$ is 
$$
\modu(\Gamma)=\inf \int_{\mathbb{C}} \rho^2 \, dA, 
$$  
where the infimum is over all \emph{admissible functions}, i.e., Borel functions $\rho:\mathbb{C} \to [0,\infty]$ satisfying $\int_\gamma \rho \, ds \geq 1$ for all locally rectifiable $\gamma \in \Gamma$. We will apply the following basic properties, see e.g. \cite[Ch. I]{Ahl66}: 
\begin{itemize} 
\item[(i)] A sense-preserving homemorphism $f:U \to V$ between subdomains of $\hatc$ is conformal if and only if every path family $\Gamma$ on $U \cap \mathbb{C}$ satisfies 
\begin{equation} 
\label{eq:invariance} 
\modu (f\Gamma)=\modu(\Gamma). 
\end{equation} 
 
\item[(ii)] If $\Gamma_1$ and $\Gamma_2$ are path families so that for every $\gamma_1 \in \Gamma_1$ there is a $\gamma_2 \in \Gamma_2$ which is a restriction of $\gamma_1$, then 
\begin{equation}\label{eq:mono} 
\modu(\Gamma_2) \geq \modu(\Gamma_1). 
\end{equation} 
\item[(iii)] If $\Gamma$ is the family of horizontal (vertical) segments connecting the vertical (horizontal) edges of rectangle $(\zeta,\zeta+t) \times (\xi,\xi+s)$, then 
\begin{equation} \label{eq:rectmodu}
\modu(\Gamma) = \frac{s}{t} \quad \quad \Big( \modu(\Gamma) = \frac{t}{s}  \Big). 
\end{equation} 
\item[(iv)] If $\Gamma$ is the family of circles $S(z_0,s)$, $s_1<s<s_2$, then  
\begin{equation} \label{eq:annmodu}
\modu(\Gamma)= \frac{\log \frac{s_2}{s_1}}{2\pi}. 
\end{equation} 
If $\Lambda$ is the family of paths joining $S(z_0,s_1)$ and $S(z_0,s_2)$, 
then 
\begin{equation} \label{eq:ann2modu}
\modu(\Lambda) = \frac{2 \pi} {\log \frac{s_2}{s_1}}. 
\end{equation}

\end{itemize}


\subsection{Neighboring disks} \label{sec:neighbors}
We say that $w=(\tilde{m},\tilde{n}n_k) \in \mathcal{W}_k$ is a \emph{bottom} if $n_k=1$ or $n_k=\frac{N_k}{2}+1$, and a \emph{top} if 
$n_k=\frac{N_k}{2}$ or $n_k=N_k$. Given $l=\pm 1$, 
we denote 
$$
w+l=(\tilde{m},\tilde{n}(n_k+l)). 
$$ 

We fix $z \in E \times F$ and $k \in \mathbb{N}$, and let $w=(\tilde{m},\tilde{n}n_k)$ be the element of $\mathcal{W}_k$ for which $z \in R^w$. The \emph{ordered} collection of the \emph{$k$:th level neighbors} of $z$ is  
\begin{equation} \label{eq:tiitos}
\mathfrak{N}_k(z)=\{D^{w-1}(\Ri),D^{w}(\Ri),D^{w+1}(\Ri),D^{w+1}(\Le),D^{w}(\Le),D^{w-1}(\Le)\} 
\end{equation} 
if $w$ is not a top or a bottom, 
\begin{equation}\label{eq:kiitos}
\mathfrak{N}_k(z)=\{D^{w}(\Ri),D^{w+1}(\Ri),D^{w+1}(\Le),D^{w}(\Le)\} 
\end{equation} 
if $w$ is a bottom, and 
$$\mathfrak{N}_k(z)=\{D^{w-1}(\Ri),D^{w}(\Ri),D^{w}(\Le),D^{w-1}(\Le)\}$$ if $w$ is a top. We determine a cyclic ordering: if 
$\mathfrak{N}_k(z)=\{D_1,\ldots,D_\ell\}$, then 
\begin{equation} \label{eq:order}
D_j<D_{j+1} \text{ for } j \in \{1,\ldots,\ell-1\}, \text{ and } D_\ell<D_1.  
\end{equation} 

We use the notation $\tau D:=\overline{\mathbb{D}}(z_0,\tau r)$ for a disk $D=\overline{\mathbb{D}}(z_0,r)$ and $\tau>0$. 

\begin{lemma} \label{lem:basic}
If $D \in \mathfrak{N}_k(z)$, then $z \in 4D$.  
\end{lemma}
\begin{proof}
Recall that the height of $R^w$ and radius of each $D \in \mathfrak{N}_k(z)$ are $2r_k$ and $r_k$, respectively. 
We assume that $w$ is not a bottom and $D=D^{w-1}(\Le)$; the other cases are proved similarly. The width of $R^{w-1}$ (and $R^{w}$) and distance between $D$ and $R^{w-1}$ are 
$2t_k$ and $t_k$, respectively. Thus, if the center of $D$ is $z_0$, then the distance between $z_0$ and the center $p_0$ of the right vertical edge $I^{w-1}(\Ri)$ of $R^{w-1}$ is $r_k+3t_k$. 
The distance between $p_0$ and the top right corner $q_0$ of $R^w$ is $3r_k+\frac{\delta_k}{2}$. Since $\max\{t_k,\delta_k \} \leq \frac{r_k}{100}$ by \eqref{eq:seasuusi}, we conclude from the Pythagorean theorem 
that   
$$
|z_0-z| \leq |z_0-q_0| \leq \Big((r_k+3t_k)^2+(3r_k+\frac{\delta_k}{2})^2 \Big)^{1/2} < 4r_k. 
$$
The proof is complete. 
\end{proof}


\subsection{Surrounding path families on $\Omega$} \label{sec:modusec}
If $A_1,A_2 \subset \mathbb{C}$ and if $U \subset \mathbb{C}$ is a domain, we say that a path $\gamma:[\alpha,\beta]\to \overline{U}$ \emph{connects $A_1$ and $A_2$ in $U$}, if $\gamma(\alpha) \in A_1$, 
$\gamma(\beta) \in A_2$, and $\gamma(t) \in U$ for all $\alpha<t<\beta$. In the following, $|\gamma|$ refers to the image of $\gamma$. We apply the notation of Section \ref{sec:neighbors} for the collection 
$\mathfrak{N}_k(z)$ of $k$:th level neighbors. 

\begin{proposition} \label{prop:surrounding} 
For every $z \in E \times F$ and $k \in \mathbb{N}$ there are families $\Gamma_j$, $j \in \{1, \ldots, \ell\}$, of paths connecting $D_j$ and $D_{j+1}\in \mathfrak{N}_k(z)$ in 
$\Omega \setminus \{\infty\}$ so that 
\begin{itemize} 
\item[(i)] $\modu(\Gamma_j) \geq \frac{1}{10}$ for every $j \in \{1, \ldots, \ell\}$, and  
\item[(ii)] if $\gamma_j \in \Gamma_j$ then $\cup_{j=1}^\ell (|\gamma_j| \cup D_j)$ separates $z$ from $\infty$.  
\end{itemize} 
\end{proposition}


\begin{figure}
    \centering
    
\begin{tikzpicture}
\draw (-3.5,.05-.04) -- (-2.5,.05-.04);
\draw (3.5,.05-.04) -- (2.5,.05-.04);
\draw (-3.5,-.05+.04) -- (-2.5,-.05+.04);
\draw (3.5,-.05+.04) -- (2.5,-.05+.04);

\draw (-2.3,.05-.04) -- (-1.3,.05-.04);
\draw (-1.1,.05-.04) -- (-.1,.05-.04);
\draw (.1,.05-.04) -- (1.1,.05-.04);
\draw (1.3,.05-.04) -- (2.3,.05-.04);

\draw (-2.3,-.05+.04) -- (-1.3,-.05+.04);
\draw (-1.1,-.05+.04) -- (-.1,-.05+.04);
\draw (.1,-.05+.04) -- (1.1,-.05+.04);
\draw (1.3,-.05+.04) -- (2.3,-.05+.04);

\draw (0,-1.2+.04)  circle (1.1); 
\draw (0,1.2-.04)  circle (1.1); 
\draw (2.4,-1.2+.04)  circle (1.1); 
\draw (2.4,1.2-.04)  circle (1.1); 
\draw (-2.4,-1.2+.04)  circle (1.1); 
\draw (-2.4,1.2-.04)  circle (1.1); 

\draw[thick,dashed] (2.4,-1.2+.04) -- (2.4,1.2-.04);
\draw[thick,dashed] (2.4,1.2-.04) -- (-2.4,1.2-.04);

\draw[thick,dashed] (-2.4,-1.2+.04) -- (-2.4,1.2-.04);
\draw[thick,dashed] (2.4,-1.2+.04) -- (-2.4,-1.2+.04);

\end{tikzpicture}

    \caption{The neighboring disks when $w$ is not a bottom or a top. The figure is rotated by $90$ degrees, and the surrounding ``chain'' consisting of the neighboring disks and segments 
    $\gamma_j \in \Gamma_j$ is illustrated by a dashed rectangle. }
    \label{fig:kuva3}
\end{figure}
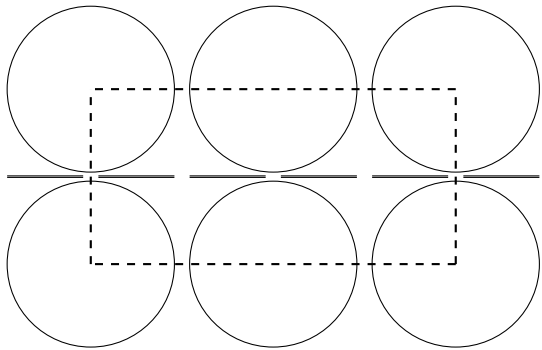


\begin{proof} 
Fix $z \in R^w$, where $w=(\tilde{m},\tilde{n}) \in \mathcal{W}_k$. Suppose first that $w$ is not a top or a bottom, and recall the cyclic order \eqref{eq:order} of elements in $\mathfrak{N}_k(z)$ 
(defined in \eqref{eq:tiitos}). See Figure \ref{fig:kuva3} for an illustration of the path families defined below. By \eqref{eq:seasuusi}, the distance $\delta_k$ between $D^{w-1}(\Ri)$ and $D^{w}(\Ri)$ is less than $\frac{r_k}{100}$, where $r_k$ is their radius. Recalling the coordinates of the 
centers of the disks from \eqref{eq:centers}, we thus conclude that every vertical segment connecting the horizontal sides of square 
$$
\Big(x^{\tilde{m}}+2t_k+\frac{9r_k}{10},x^{\tilde{m}}+2t_k+\frac{11r_k}{10}\Big)\times \Big(y^{\tilde{n}}-\frac{r_k}{10},y^{\tilde{n}}+\frac{r_k}{10} \Big) 
$$
contains a subsegment joining $D^{w-1}(\Ri)$ and $D^{w}(\Ri)$ in $\Omega \setminus \{\infty\}$. For the family $\Gamma_1$ of such subsegments,  
\eqref{eq:mono} and \eqref{eq:rectmodu} yield  
\begin{equation} \label{eq:ekaarvio} 
\modu(\Gamma_1) \geq 1 \geq \frac{1}{10}. 
\end{equation} 
We define families $\Gamma_2$, $\Gamma_4$ and $\Gamma_5$ of vertical segments in a similar manner, and apply the argument above to show that \eqref{eq:ekaarvio} 
also holds for such families. 

We next define family $\Gamma_3$ of horizontal segments connecting $D^{w+1}(\Ri)$ and $D^{w+1}(\Le)$. We recall from \eqref{eq:centers} that the points 
minimizing the distance of these disks are 
$$
x^{\tilde{m}} \pm 2t_k +i(y^{\tilde{n}}+2r_k+\delta_k)=: x^{\tilde{m}} \pm 2t_k +iy^+. 
$$
In particular, by \eqref{eq:seasuusi} we have that 
$$ 
\operatorname{dist}(D^{w+1}(\Ri), D^{w+1}(\Le))=4t_k \leq \frac{\delta_k}{25}. 
$$
Recalling that the set 
$$
T^{\tilde{n}+1}:=\Big\{x+iy: \, y^+ - \frac{\delta_k}{2}<y<y^+ + \frac{\delta_k}{2}\Big\} 
$$ 
does not intersect $E \times F$, we conclude that every horizontal segment connecting the vertical sides of square 
$$
(x^{\tilde{m}}-10t_k,x^{\tilde{m}}+10t_k)\times (y^+ - 10t_k,y^+ +10t_k) \subset T^{\tilde{n}+1}
$$ 
contains a subsegment connecting $D^{w+1}(\Ri)$ and $D^{w+1}(\Le)$ in $\Omega \setminus \{\infty\}$. Applying \eqref{eq:mono} and \eqref{eq:rectmodu}, we see that 
\eqref{eq:ekaarvio} holds for the family $\Gamma_3$ of all such subsegments. The same argument applies to $\Gamma_6$. We conclude that the desired modulus bounds (i) hold. 
Condition (ii) follows directly from the definitions of families $\Gamma_j$. We have established the proposition when $w$ is not a top or a bottom. 

We now assume that $w$ is a bottom, and recall the cyclic order \eqref{eq:order} of elements in $\mathfrak{N}_k(z)$ (defined in \eqref{eq:kiitos}). 
We can define families $\Gamma_1,\Gamma_2$ of vertical segments and $\Gamma_3$ of horizontal segments as above so that \eqref{eq:ekaarvio} holds. See Figure 
\ref{fig:kuva4} for an illustration. 

We define the final family $\Gamma_4$ of circular arcs as follows: Since $100 t_k \leq r_k$ by \eqref{eq:seasuusi}, for every circle $S_r$, with radius 
$\frac{3r_k}{4}<r < \frac{3r_k}{2}$ centered at $x^{\tilde{m}}+i(y^{\tilde{n}}-r_k)$, there is a connected component of $S_r \setminus (D^w(\Le) \cup D^w(\Ri))$ 
whose closure $\eta(r)$ contains the lower semicircle of $S_r$ and connects $D^w(\Le)$ and $D^w(\Ri)$. Moreover, since $w$ is a bottom and 
$100r_k \leq \delta_{k-1}$ by \eqref{eq:kaksuusi}, $\eta(r)$ does not intersect any other complementary components of $\Omega$. 

We conclude that each $\eta(r)$ connects $D^w(\Le)$ and $D^w(\Ri)$ in $\Omega \setminus \{\infty\}$. Applying \eqref{eq:mono} and \eqref{eq:annmodu} shows that the family 
$\Gamma_4$ of such arcs satisfies 
$$
\modu(\Gamma_4) \geq \frac{\log 2}{2\pi} \geq \frac{1}{10}. 
$$
We have proved the desired modulus bounds (i) for bottoms. Tops are treated similarly. Condition (ii) follows again from the definitions of families $\Gamma_j$. 
The proof is complete. 
\end{proof}

\begin{figure}
    \centering
    
\begin{tikzpicture}
\draw (-6.5,.05-.04) -- (-5.5,.05-.04);
\draw (3.5,.05-.04) -- (2.5,.05-.04);
\draw (-6.5,-.05+.04) -- (-5.5,-.05+.04);
\draw (3.5,-.05+.04) -- (2.5,-.05+.04);

\draw (-5.3,.05-.04) -- (-4.3,.05-.04);
\draw (-1.1,.05-.04) -- (-.1,.05-.04);
\draw (.1,.05-.04) -- (1.1,.05-.04);
\draw (1.3,.05-.04) -- (2.3,.05-.04);

\draw (-5.3,-.05+.04) -- (-4.3,-.05+.04);
\draw (-1.1,-.05+.04) -- (-.1,-.05+.04);
\draw (.1,-.05+.04) -- (1.1,-.05+.04);
\draw (1.3,-.05+.04) -- (2.3,-.05+.04);

\draw (0,-1.2+.04)  circle (1.1); 
\draw (0,1.2-.04)  circle (1.1); 
\draw (2.4,-1.2+.04)  circle (1.1); 
\draw (2.4,1.2-.04)  circle (1.1); 
\draw (-5.4,-1.2+.04)  circle (1.1); 
\draw (-5.4,1.2-.04)  circle (1.1);

\draw[thick,dashed] (2.4,-1.2+.04) -- (2.4,1.2-.04);
\draw[thick,dashed] (2.4,1.2-.04) -- (-1.2,1.2-.04);
\draw[thick,dashed] (2.4,-1.2+.04) -- (-1.2,-1.2+.04);

\draw[thick,dashed] (-1.2,1.2-.04) arc [start angle=90, end angle=270, radius=1.2-.04];

\fill (-6.7,1.2-.04) circle (1pt);
\fill (-6.9,1.2-.04) circle (1pt);
\fill (-7.1,1.2-.04) circle (1pt);

\fill (-6.7,-1.2+.04) circle (1pt);
\fill (-6.9,-1.2+.04) circle (1pt);
\fill (-7.1,-1.2+.04) circle (1pt);

\fill (3.7,1.2-.04) circle (1pt);
\fill (3.9,1.2-.04) circle (1pt);
\fill (4.1,1.2-.04) circle (1pt);

\fill (3.7,-1.2+.04) circle (1pt);
\fill (3.9,-1.2+.04) circle (1pt);
\fill (4.1,-1.2+.04) circle (1pt);

\end{tikzpicture}

    \caption{The neighboring disks when $w$ is a bottom. The figure is rotated by $90$ degrees, and the surrounding ``chain'' consisting of the neighboring disks and paths 
    $\gamma_j \in \Gamma_j$ is illustrated by a dashed loop. }
    \label{fig:kuva4}
\end{figure}
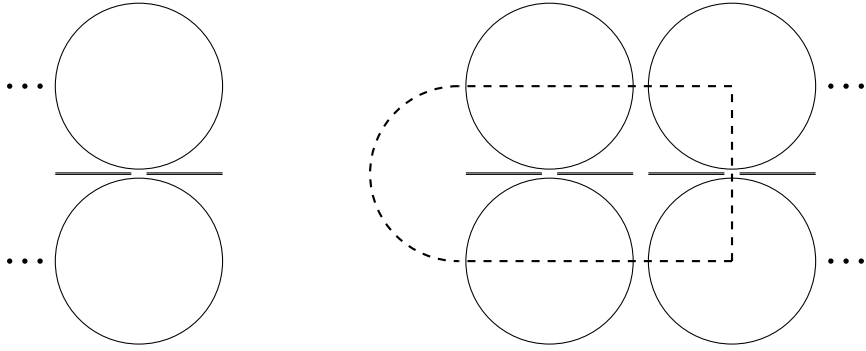



\subsection{Distortion estimate} \label{sec:distortion}
Given a domain $G \subset \hatc$, we denote by $\mathcal{C}(G)$ the collection of connected components of $\hatc \setminus G$ and by $\hat{G}$ the quotient space $\hatc / \sim$, 
where 
$$
x \sim y \text{ if either } x=y \in G \text{ or } x,y \in p \text{ for some } p \in \mathcal{C}(G). 
$$
The corresponding quotient map is $\pi_G:\hatc \to \hat{G}$. Identifying each $x \in G$ and $p \in \mathcal{C}(G)$ with $\pi_G(x)$ and $\pi_G(p)$, respectively, we have $\hat{G}=G \cup \mathcal{C}(G)$.  
A homeomorphism $f:G \to G'$ has a homeomorphic extension $\hat{f}:\hat{G} \to \hat{G'}$; see e.g. \cite{NtaYou20}*{Section 3} for a detailed discussion. 

Let $f:\Omega \to \Omega'$ be a conformal map onto a circle domain $\Omega'$. We eventually want to conclude that $f$ is a M\"obius transformation. Post-composing $f$ with another 
M\"obius transformation does not affect the conclusion, so we may assume that $f(\infty)=\infty$. Recall the notation $\tau D:=\mathbb{D}(z_0,\tau r)$ for a disk $D=\mathbb{D}(z_0,r)$ and $\tau>0$. 
We denote the radius of a disk $D$ by $r(D)$. 

We continue to apply the notation of Section \ref{sec:neighbors} for the collection $\mathfrak{N}_k(z)$ of $k$:th level neighbors.  

\begin{proposition} \label{prop:fclose}
For every $z \in E \times F$ and $k \in \mathbb{N}$ there is 
\begin{equation} \label{eq:inclusion}
D^k_z \in \mathfrak{N}_k(z) \quad \text{so that} \quad z \in 4D^k_z \quad \text{and} \quad \hat{f}(\{z\}) \subset 10^{30} \hat{f}(D^k_z).  
\end{equation}  
\end{proposition}

\begin{proof} Suppose $z \in R^w$, where $w=(\tilde{m},\tilde{n}) \in \mathcal{W}_k$. We recall that the first inclusion in 
\eqref{eq:inclusion} holds for all $D_j \in \mathfrak{N}_k(z)$ by Lemma \ref{lem:basic}. 

Given $j \in \{1,\ldots, \ell\}$, we denote $D'_j=\hat{f}(D_j)$ and $\Gamma'_j=f(\Gamma_j)$. Here $\Gamma_j$ is the path family in Proposition \ref{prop:surrounding}, which together 
with the conformal invariance of modulus (see \eqref{eq:invariance}) yields  
\begin{equation} \label{eq:falaraja}
\modu(\Gamma'_j) \geq \frac{1}{10} \quad \text{for every } j \in \{1,\ldots,\ell\}. 
\end{equation} 
We now claim that for every $j \in \{1,\ldots,\ell\}$ there is a $\gamma'_j \in \Gamma'_j$ for which 
\begin{equation}\label{eq:shortcurve}
\operatorname{diam}(|\gamma'_j|)  \leq 2\cdot(\exp(20 \pi)+1) r(D'_j) \leq 2\cdot(\exp(20 \pi)+1) r', 
\end{equation} 
where $r'=\max \Big\{r(D'_j): \, j \in \{1,\ldots,\ell\} \Big\}$. The second inequality is trivial. To prove the first inequality, recall that every $\gamma \in \Gamma_j'$ intersects 
$D_j'$ by the definition of $\Gamma_j'$. If the first inequality in \eqref{eq:shortcurve} fails then every $\gamma \in \Gamma_j'$ also intersects the boundary of $(\exp(20 \pi)+1)D'_j$. 
Combining \eqref{eq:falaraja} with the modulus identity \eqref{eq:ann2modu} and monotonicity property \eqref{eq:mono}, we see that 
$$
\frac{1}{10} \leq \modu(\Gamma_j') \leq \frac{2\pi}{\log(\exp(20\pi)+1)}, 
$$
which is a contradiction. We have proved \eqref{eq:shortcurve}. 

Let $\gamma'_j$, $j \in \{1,\ldots,\ell\}$, be the paths in \eqref{eq:shortcurve}. Part (ii) of Proposition \ref{prop:surrounding} shows that  
$$
T:=\cup_{j=1}^\ell (D'_j \cup |\gamma'_j|) \quad \text{separates } \hat{f}(\{z\}) \text{ from infinity}.   
$$ 
In particular, the distance between $D_j'$ and any point $z' \in \hat{f}(\{z\})$ is bounded from above by the diameter of $T$ (we will soon show that $z'$ is unique). 
Applying \eqref{eq:shortcurve} and triangle inequality, we conclude that if $D \in \mathfrak{N}_k(z)$ is one of the neighbors satisfying $r(D')=r'$, where $D'=\hat{f}(D)$, then  
\begin{eqnarray*} 
\dist(D',z') &\leq& \diam(T) \leq \sum_{j=1}^\ell (2r(D'_j)+\operatorname{diam}(|\gamma_j'|))  \\ 
&\leq& 12\cdot(2+\exp(20 \pi))r' \leq 10^{29}r';    
\end{eqnarray*} 
recall that the cardinality $\ell$ of $\mathfrak{N}_k(z)$ is at most $6$. We conclude that also the second inclusion in \eqref{eq:inclusion} holds if $D^k_z=D'$. The proof is complete. 
\end{proof} 


\subsection{Boundary extension of $f$} \label{sec:boundary}
We continue to investigate a conformal map $f:\Omega \to \Omega'$ onto a circle domain $\Omega'$ satisfying $f(\infty)=\infty$. We prove that $f$ extends to a homeomorphism between the closures of $\Omega$ and $\Omega'$.  

\begin{lemma}\label{lemma:extension}
The map $f$ extends to a homeomorphism $\tilde{f}:\overline{\Omega} \to \overline{\Omega'}$. 
\end{lemma}

\begin{proof}
Since the disk components $D$ of $\hatc \setminus \Omega$ are isolated in $\pi_\Omega(\mathcal{C}(\Omega))$, Ca\-ra\-th\'eo\-do\-ry's theorem shows that $f$ has a homeomorphic extension $\partial D \to \partial \hat{f}(D)$. Hence it suffices to show that $\operatorname{diam}(\hat{f}(\{z\}))=0$ for every $z \in E \times F$. 

Notice that the preimage of $\hat{f}(\mathcal{C}(\Omega))$ under $\pi_{\Omega'}^{-1}$ is bounded in $\mathbb{C}$ since $\infty \in \Omega'$. Consequently, the sequence of disks $D_z^k$ in 
Proposition \ref{prop:fclose} satisfies $r(\hat{f}(D_z^k)) \to 0$ as $z \to \infty$. Combining with \eqref{eq:inclusion}, we have 
$$
\operatorname{diam}(\hat{f}(\{z\})) \leq 2\cdot 10^{30} \lim_{k \to \infty} r(\hat{f}(D_z^k)) =0.  
$$ 
The proof is complete. 
\end{proof} 

We now apply the distortion bounds in Proposition \ref{prop:fclose} towards regularity of the extension $\tilde{f}$ on $E \times F$. 
Recall that the $1$-dimensional \emph{Hausdorff content} $\mathcal{H}_\infty^1(A)$ and \emph{Hausdorff measure} $\mathcal{H}^1(A)$ of $A \subset \mathbb{C}$ are defined by 
\begin{eqnarray*}
\mathcal{H}_\delta^1(A) &=& \inf \left\{ \sum_{j=1}^{\infty} \operatorname{diam}(V_j): \, A \subset \bigcup_j V_j, \, \operatorname{diam}(V_j) <\delta \right\}, \quad 0<\delta \leq \infty, \\
\mathcal{H}^1(A) &=& \lim_{\delta \to 0} \mathcal{H}_\delta^1(A)=\sup_{\delta > 0} \mathcal{H}_\delta^1(A). 
\end{eqnarray*}
We say that a property 
holds for \emph{almost every path in $\mathbb{C}$}, if the family of paths for which the property does not hold has zero modulus.

\begin{proposition} \label{prop:abscont} 
Let $\tilde{f}$ be the extension in Lemma \ref{lemma:extension}. Then almost every path $\gamma$ in $\mathbb{C}$ satisfies 
\begin{equation} \label{eq:abscont}
\mathcal{H}^1\Big(\tilde{f}\big(|\gamma|\cap (E \times F)\big)\Big)=0. 
\end{equation} 
\end{proposition}

We will not apply Proposition \ref{prop:abscont} later in this paper, and it could be proved by applying \cite[Theorem 1.2]{Nta24a} to the extension $f_*$ in Proposition \ref{prop:Schottky} below. We give a proof of the proposition for the reader's convenience, since it is short and contains the basic idea on how distortion bounds are applied to prove the regularity properties  which eventually lead to the rigidity of 
$\Omega$. See Remark \ref{rem:viime} for further discussion. We denote 
$$ 
\mathcal{D}_k= \Big\{D=D^w(\Le) \text{ or }D^w(\Ri): \, w \in \mathcal{W}_k \Big\}, \quad \mathcal{D}=\bigcup_{k=1}^\infty \mathcal{D}_k. 
$$ 

\begin{proof}[Proof of Proposition \ref{prop:abscont}]
Recall that $\infty \in \Omega'$. We notice that if   
\begin{eqnarray*}  
\rho(z)=\left\{ \begin{array}{ll} 
\frac{r(\hat{f}(D))}{r(D)}, & z \in D  \text{ for some } D \in \mathcal{D}, \\ 
0, & \text{otherwise},  
\end{array}
\right. 
\end{eqnarray*} 
then $\rho \in L^{2}(\mathbb{C})$. 
We fix $k \in \mathbb{N}$ and a rectifiable path $\gamma$ in $\mathbb{C}$, and denote $|\gamma| \cap (E \times F)=:G$. Then Proposition \ref{prop:fclose} guarantees that if $\tau=10^{30}$ and  
$$
\mathcal{G}_k=\{D \in \mathcal{D}_k: 4D \cap G \neq \emptyset \}=\{D_1,\ldots,D_p\}, 
$$
then $\cup_{j=1}^p (\tau \hat{f}(D_j))$ covers $\tilde{f}(G)$. Thus, denoting $D'_j=\hat{f}(D_j)$ we have  
\begin{equation} \label{eq:possu}
\mathcal{H}^1_{\infty}(\tilde{f}(G)) \leq 2\tau \sum_{j=1}^p r(D'_j). 
\end{equation} 
Increasing $k$ if necessary, we may assume that $\operatorname{diam}(|\gamma|) \geq r_k$, where $r_k$ is the common radius of disks $D_j \in \mathcal{G}_k$. It then follows from the definition of $\mathcal{G}_k$ that we have 
$\mathcal{H}^1(|\gamma| \cap 5D_j) \geq r_k$, so that 
\begin{equation}  
\label{eq:tossu} r(D'_j) \leq 25 r_k \inf\{\mathcal{M}(\rho)(z): \, z \in |\gamma| \cap 5D_j\}  
\leq {25} \int_{|\gamma|} \chi_{5D_j}\mathcal{M}(\rho) \, d\mathcal{H}^1 
\end{equation}
for every $1 \leq j \leq p$, where $\mathcal{M}(\rho)$ is the non-centered maximal function of $\rho$ (see e.g. \cite[Ch. 2]{Hei01}).  
Combining \eqref{eq:possu} and \eqref{eq:tossu}, we conclude that 
\begin{equation}\label{eq:kossu}
\mathcal{H}^1_{\infty}(\tilde{f}(G)) \leq 50\tau \int_{|\gamma|} \mathcal{M}(\rho) \sum_{j=1}^p \chi_{5D_j} \, d\mathcal{H}^1 \leq 50\tau \int_{\gamma} \mathcal{M}(\rho) \sum_{j=1}^p \chi_{5D_j} \, ds. 
\end{equation}
Since disks $D_j$ are disjoint we have $\sum_{j=1}^p \chi_{5D_j}(z) \leq 36$ for every $z \in \mathbb{C}$. Thus \eqref{eq:kossu} yields 
\begin{equation} \label{eq:mi}
\mathcal{H}^1_{\infty}(\tilde{f}(G)) \leq  1800 \tau  \int_{\gamma} \chi_{N_{9r_k}(E \times F)}\mathcal{M}(\rho) \, ds, 
\end{equation} 
where $N_{s}(U)$ is the closed $s$-neighborhood of $U \subset \mathbb{C}$. 

Recall that $\mathcal{M}(\rho) \in L^2(\mathbb{C})$ by the $L^2$-boundedness of maximal functions. By the definition of conformal modulus it then follows that 
$\mathcal{M}(\rho)$ is integrable over almost every path $\gamma$ in $\mathbb{C}$. Thus, we can apply \eqref{eq:mi} and monotone convergence with $k \to \infty$ to show that 
\begin{equation} \label{eq:egypt}
\mathcal{H}^1_{\infty}(\tilde{f}(G)) \leq  1800 \tau  \int_{\gamma} \chi_{E \times F}\mathcal{M}(\rho) \, ds 
\end{equation} 
for almost every path $\gamma$ in $\mathbb{C}$. On the other hand, since $E\times F$ has Lebesgue measure zero by Lemma \ref{lem:areazero}, we have that 
$\int_{\mathbb{C}} \chi_{E \times F}\mathcal{M}(\rho) \, dA=0$. In particular, $\int_{\gamma} g\, ds =0$ for almost every $\gamma$ 
in $\mathbb{C}$, again by the definition of conformal modulus. Since $\mathcal{H}^1(\tilde{f}(G))=0$ precisely when $\mathcal{H}^1_{\infty}(\tilde{f}(G))=0$ (see e.g. \cite[Ch. 8.3]{Hei01}), our claim \eqref{eq:abscont} follows from \eqref{eq:egypt}. 
\end{proof}


\subsection{Global extension of $f$} \label{sec:Schottky} 
In this section we apply Lemma \ref{lemma:extension} and repeated reflections across the boundary circles of $\Omega$ and $\Omega'$ to extend $f$ to all of $\hatc$. The method has been applied to prove rigidity of circle domains and Schottky sets e.g. in \cite{HeSch94}, \cite{BKM09}, \cite{You16}, \cite{NtaYou20}, \cite{NtaRaj23}. We refer to \cite[Section 7.1]{NtaYou20} 
for the details of the following construction and the proof of Proposition \ref{prop:Schottky} below.  

As before, we assume that $f(\infty)=\infty$. Given $D \in \mathcal{D}$, let $R_D$ be the reflection across the circle $\partial D$. 
The \emph{Schottky group} $\mathcal{S}(\Omega)$ is the free discrete group generated by $\{R_D: \, D \in \mathcal{D}\}$. Every non-identity element $g$ of $\mathcal{S}(\Omega)$ 
can be uniquely written as 
\begin{equation} \label{eq:paki}
g=R_{D_1} \circ \cdots \circ R_{D_\ell}, \quad \text{where } D_{j+1} \neq D_j \text{ for all } 1 \leq j \leq \ell-1. 
\end{equation} 

Denoting $D'=\hat{f}(D)$, the map $f$ admits a conformal extension 
\begin{eqnarray*} 
f_*:\Omega \cup \bigcup_{D \in \mathcal{D}}(R_D(\Omega) \cup \partial D) \to \Omega' \cup \bigcup_{D \in \mathcal{D}}(R_{D'}(\Omega') \cup \partial D'): 
\end{eqnarray*} 
we set $f_*(z)=(R_{D'}\circ f \circ R_D)(z)$ for $z \in R_D(\Omega)$ and apply Lemma \ref{lemma:extension} and the Schwarz reflection principle to extend $f_*$  
across the boundary circle $\partial D$. Continuing inductively, we see that $f_*$ can be further extended to the union $\Omega_*$ of sets 
$g(\Omega) \cup g(\partial D_\ell)$, $g \in \mathcal{S}(\Omega)$. 
Here $D_\ell$ is the disk in \eqref{eq:paki}.  

Thus, we have a conformal homeomorphism $f_*:\Omega_* \to \Omega_*'$, where $\Omega_*'$ is defined as $\Omega_*$ but using elements $g'=\mathcal{S}(\Omega')$ instead 
of elements $g=\mathcal{S}(\Omega)$. The boundary satisfies $\partial \Omega_*=\hatc \setminus \Omega_*=X \cup Y$, where  
\begin{equation}
\label{eq:ahja1} 
X = \bigcup_{g \in \mathcal{S}(\Omega)} g(E \times F), 
\end{equation} 
and for every $z \in Y$ there are disks $D_j \in \mathcal{D}$ so that if $B_0=D_0$ then 
\begin{equation} 
\label{eq:ahja2} 
\{z\} = \cap_{j=0}^\infty B_j, \quad \text{where} \quad B_{j+1}=R_{D_1} \circ \cdots \circ R_{D_j}(D_{j+1}) \subset B_j.  
\end{equation} 

Given $z \in Y$, the intersection of the disks $B_j'$ bounded by circles $f_*(\partial B_j)$ is a point $z' \in \partial \Omega_*'$. Applying Lemma \ref{lemma:extension} to $X$ and setting $f_*(z)=z'$ for $z \in Y$ shows that $f_*$ has a homeomorphic extension to $\hatc$, see 
\cite[Lemma 7.5]{NtaYou20}. The following proposition summarizes our discussion.  

\begin{proposition}\label{prop:Schottky}
The map $f_*:\hatc \to \hatc$ is a homeomorphism, and conformal on $\Omega_*=\hatc \setminus (X \cup Y)$. Here $X$ and $Y$ satisfy \eqref{eq:ahja1} and \eqref{eq:ahja2}, 
respectively.   
\end{proposition} 


\subsection{Eccentric distortion and conformality}\label{sec:final}

The final step in the proof of the rigidity of $\Omega$ is showing that the homeomorphic extension $f_*:\hatc \to \hatc$ of $f$ (see Proposition \ref{prop:Schottky}) is conformal on $\hatc$, and therefore a M\"obius transformation. We now show how conformality follows from Proposition \ref{prop:fclose} and a recent characterization of conformal maps by Ntalampekos \cite{Nta24a}. 

Recall that a sense-preserving homemorphism $h:G \to G'$ between subdomains of $\hatc$ is \emph{$K$-quasiconformal}, $K \geq 1$, if the conformal modulus of every path family $\Gamma$ 
in $G \cap \mathbb{C}$ satisfies 
\begin{equation} \label{eq:pask}
\frac{1}{K} \modu(\Gamma) \leq \modu(h\Gamma) \leq K \modu(\Gamma). 
\end{equation} 
By \eqref{eq:invariance}, $1$-quasiconformality is equivalent with conformality. 

The classical \emph{metric definition of quasiconformality} (see e.g. \cite[Ch. 4]{Vai71}) is given in terms of the \emph{metric distortion}, which at a point $z_0$ measures the distortion of 
images under $h$ of small disks centered at $z_0$. We apply a more flexible notion of metric distortion which was recently introduced by Ntalampekos: We say that the \emph{eccentricity} of a bounded open set $A \subset \mathbb{C}$ is 
$$
E(A)=\inf\{M \geq 1: \, \text{there exists an open disk } B \text{ such that } B \subset A \subset MB \}. 
$$
The \emph{eccentric distortion} of a topological embedding $h:G \to \hatc$ of an open $G \subset \hatc$ at $z_0 \in G \setminus (\{\infty,h^{-1}(\infty)\})$ is 
\begin{eqnarray*} 
E_h(z_0)&=&\inf \{M \geq 1: \, \text{there exists a sequence of open sets } A_k \subset G, \,   \\ & &k\in \mathbb{N}, \text{ with } z_0 \in A_k, \, k \in \mathbb{N}, \text{ and } \operatorname{diam}(A_k) \to 0 
\text{ as } k \to \infty  \\ & &  \text{ such that } E(A_k) \leq M \text{ and } E(h(A_k)) \leq M \text{ for all } k \in \mathbb{N}\}. 
\end{eqnarray*}
The definition can be extended to $\{\infty,h^{-1}(\infty)\}$ by composing $h$ with M\"obius transformations. 

\begin{theorem}\label{thm:dimitrios}
Let $G\subset \hatc$ be open and $h:G \to \hatc$ a sense-preserving topological embedding. If there is $H \geq 1$ so that 
\begin{equation}\label{eq:paupau}
E_h(z_0) \leq H \quad \text{for all } z_0 \in G,  
\end{equation} 
then $h$ is quasiconformal on $G$. If in addition to \eqref{eq:paupau} also $E_h(z_0)=1$ for almost every $z_0 \in G$, then $h$ is conformal on $G$. 
\end{theorem} 
\begin{proof} 
The first claim is \cite[Theorem 1.2]{Nta24a}, and the second claim is \cite[Lemma 2.5]{Nta23a}.  
\end{proof} 

We can now finish the proof of Theorem \ref{thm:rigidity}. We need to prove that 
the map $f_*$ in Proposition \ref{prop:Schottky} is conformal on $\hatc$. Conformality of $f_*$ in $\Omega_*$ shows that $E_{f_*}(z_0)=1$ 
for every $z_0\in \Omega_*$. We also have $E_{f_*}(z_0)=1$ for every $z_0 \in Y$, since we can apply the interiors of the disks $B_j$ in \eqref{eq:ahja2} to test the definition of the 
eccentric distortion. 

Finally, we claim that 
\begin{equation} 
\label{eq:shalala}
E_{f_*}(z_0) \leq 10^{30} \quad \text{for every } z_0 \in X.  
\end{equation} 
If $z_0 \in E \times F$, we can test the definition of eccentric distortion with arbitrarily small open neighborhoods of the unions of $z_0$ and the neighbors $D^k_z$ in Proposition \ref{prop:fclose} to prove \eqref{eq:shalala}. 

If $g \in \mathcal{S}(\Omega)$ is a non-identity element, then every $z_0 \in g(E \times F)$ has an open neighborhood $U$ so that 
$$
f_*(z)=(g' \circ f_* \circ g^{-1})(z) \quad \text{for all } z \in U, 
$$
where $g' \in \mathcal{S}(\Omega')$. The eccentric distortion is not affected by compositions with $g'$ and $g^{-1}$, because they are (anti)conformal. Since 
$E_{f_*}(g^{-1}(z_0)) \leq 10^{30}$ by the previous paragraph, we conclude that \eqref{eq:shalala} holds for all $z_0 \in X$.  
The proofs of Theorem \ref{thm:rigidity} and Theorem \ref{thm:main} are complete. 

\begin{remark} \label{rem:viime}
The main issue in proving Theorem \ref{thm:dimitrios} is the Sobolev regularity of the map $h$. In our setting this amounts to proving absolute continuity of $f_*$ on almost every path 
across the exceptional sets $X$ and $Y$. Set $X$ could be handled using Proposition \ref{prop:abscont} (see also \cite[Theorem 4.1]{NtaRaj23}). A similar approach can also be applied to $Y$, but the arguments are much more involved and include the covering theorems introduced in \cite{Nta24a}.   
\end{remark}


\subsection*{Acknowledgment}
We thank Dimitrios Ntalampekos for his comments.

\vskip 10pt
\noindent

\bibliographystyle{alpha}
	\bibliography{Bibliography}

\end{document}